\documentclass[11pt]{amsart}
\usepackage{geometry}                
\usepackage{graphicx}
\usepackage{amssymb}
\usepackage{amsthm}
\usepackage{epstopdf}
\usepackage{hyperref}
\newtheorem{theorem}{Theorem}[section]
\newtheorem{corollary}[theorem]{Corollary}

\newtheorem{definition}[theorem]{Definition}
\newtheorem{conjecture}{Conjecture}[section]

\numberwithin{equation}{section}

\title{The additive divisor problem and minorants of divisor functions}
\author{J.C. Andrade and K. Smith }

\date{\today}                                           

\address{Department of Mathematics, University of Exeter, Exeter, EX4 4QF, UK}
\email{j.c.andrade@exeter.ac.uk}
\email{ks614@exeter.ac.uk}

\subjclass[2010]{Primary 11N37; Secondary 11M06}
\keywords{divisor function, additive divisor sum, arithmetic progression, uniform distribution, exponent of distribution}

\begin{document}
\maketitle
\begin{abstract}
In this paper we consider the general additive divisor problem. Here the divisor functions $d_k(n)$ are the number of ways of writing a natural number $n$ as a product of $k$ factors, and the problem is that of establishing asymptotic formulae for the correlations $\sum_{n\leq x}d_k(n)d_{\ell}(n+h)$ with $h,k,\ell\in\mathbb{N}$. We show that the conjectured asymptotic formulae hold when one or both of the divisor functions are replaced by the minorants $d_k(n,A)=\sum_{m|n,m\leq n^A}d_{k-1}(m)$ with $A$ sufficiently small, leading us to obtain new lower bounds for the asymptotics in the original problem. The main arguments rest on a study of the distribution of the minorants $d_k(n,A)$ in arithmetic progressions. 
\end{abstract}


\section{Introduction and statement of results}
In this paper we consider the problem of finding asymptotic formulae for additive divisor sums
\begin{eqnarray}\label{def}
D_{h,k,{\ell}}(x)=\sum_{n\leq x}d_k(n+h)d_{\ell}(n)
\end{eqnarray}
where  $h,k,{\ell}\in\mathbb{N}$ are fixed strictly positive natural numbers and $d_k(n)$ denotes the number of ways of writing $n$ as a product of $k$ factors. In other words, this is the problem of counting the number of solutions of the Diophantine equation
\begin{eqnarray}
h=n_1\cdots n_k-m_1\cdots m_{\ell} \nonumber
\end{eqnarray} 
where $(m_1,...,m_{\ell})\in\mathbb{N}^l$, $(n_1,...,n_k)\in\mathbb{N}^k$ and $n_1\cdots n_k\leq x$  as $x\rightarrow\infty$. \\

When $k={\ell}$, the correlations in (\ref{def}) arise in connection with the problem of finding asymptotic formulae for the $2k$th moments of the Riemann zeta function on the critical line.
 This connection was first exploited by Ingham \cite{Ing} in the course of proving his asymptotic formula for the fourth moment. Ingham proved that
\begin{eqnarray}
D_{h,2,2}(x)\sim\frac{6}{\pi^2}\sigma_{-1}(h)\log^2x \nonumber
\end{eqnarray}
where $\sigma_{z}(n)=\sum_{d|n}d^{z}$,
and subsequently Estermann \cite{Est} established the asymptotic expansion
\begin{eqnarray}\label{est}
D_{h,2,2}(x)=xP_{h,2,2}(\log x)+O\left(x^{11/12+\epsilon}\right)
\end{eqnarray}
where $P_{h,2,2}$ is a polynomial of degree $2$. Estermann demonstrated that $D_{h,2,2}(x)$ is related to the spectral theory of modular forms, his result making use of a non-trivial bound for Kloosterman sums. Subsequently, Heath-Brown \cite{Heath} used Weil's bound \cite{weil} for Kloosterman sums to obtain the error term $O\left(x^{5/6+\epsilon}\right)$ in (\ref{est}), which was later improved  by Motohashi \cite{mot} to $O\left(x^{2/3+\epsilon}\right)$ uniformly for $h\leq x^{20/27}$. Each of these improvements lead to corresponding improvements of the error term in the asymptotic expansion for the fourth moment. \\

When $k\in\mathbb{N}$ and $\ell=2$, the additive divisor problem is tractable. Many authors have worked on this, for instance Hooley \cite{H}, Linnik \cite{lin}, Heath-Brown \cite{H2}, Motohashi \cite{moto}, Bykovski and Vinogradov \cite{by} and Topacogullari \cite{to2,to1,to3}, among others. Thus for any fixed $k$ it is known that there is a $\delta=\delta(k)>0$ and a polynomial $P_{h,k,2}$ of degree $k$ such that
\begin{eqnarray}\label{adv}
D_{h,k,2}(x)=xP_{h,k,2}(\log x)+O_{h,k}(x^{1-\delta}).
\end{eqnarray}
Despite these advances, the problem remains intractable when both $k,{\ell}\geq 3$. The main conjecture is as follows.

\begin{conjecture}\label{manc} If $h,k,{\ell}\in\mathbb{N}$ with $k,{\ell}$ fixed and $h=O(x^{1-\epsilon})$ for some fixed $\epsilon>0$, then there is a fixed $\delta=\delta(k,{\ell},\epsilon)>0$ and a polynomial $P_{h,k,{\ell}}$ of degree $k+{\ell}-2$ such that
\begin{eqnarray}\label{mainp}
D_{h,k,{\ell}}(x)=xP_{h,k,{\ell}}(\log x)+O_{\epsilon,k,{\ell}}(x^{1-\delta}).\nonumber
\end{eqnarray}
The asymptotic is conjectured to be
\begin{eqnarray}\label{lead}
\frac{D_{h,k,{\ell}}(x)}{x\log^{k+{\ell}-2}x}\sim\frac{C_{k,{\ell}}f_{k,{\ell}}(h)}{(k-1)!({\ell}-1)!}
\end{eqnarray}
as $x\rightarrow\infty$, where\footnote{The general form of the coefficients $C_{k,{\ell}}$ and $f_{k,{\ell}}(h)$ appearing in (\ref{lead}) were calculated by Ng and Thom \cite{NgT} based on the techniques introduced by Conrey and Gonek \cite{CG}. The same prediction was made by Tao \cite{Tao} based on pseudorandomness heuristics.}
\begin{eqnarray}\label{c}
C_{k,{\ell}}=\prod_p \left(1-p^{-1}\right)^{{\ell}-1}+\left(1-p^{-1}\right)^{k-1}-\left(1-p^{-1}\right)^{k+{\ell}-2}
\end{eqnarray}
and
\begin{eqnarray}\label{fform}
f_{k,{\ell}}(h)=\prod_{p|h} \frac{(1-p^{-1})\sum_{\alpha=0}^{\gamma}d_{{\ell}-1}(p^{\alpha})  
\sum_{\beta=\alpha}^{\infty}d_k(p^{\beta})p^{-\beta}+d_k(p^{\gamma})\sum_{\alpha=\gamma+1}^{\infty}d_{{\ell}-1}(p^{\alpha}) p^{-\alpha}}{(1-p^{-1})^{1-k}+(1-p^{-1})^{1-{\ell}}    -1} \nonumber\\
\end{eqnarray}
where  $h=\prod p^{\gamma}$.
\end{conjecture} 
\noindent Conjecture \ref{manc} was formulated by 
Conrey and Gonek \cite{CG} and Iv\' ic \cite{I1,i2} using the `$\delta$-method' of Duke, Friedlander and Iwaniec \cite{DFI} and recently refined by  Ng and Thom \cite{NgT} and  Tao \cite{Tao}. The full polynomial $P_{h,k,{\ell}}$ (i.e. including the lower order terms) is also described in \cite{NgT} and \cite{I1}. The relationship between Conjecture \ref{manc} in the cases $k=\ell$ and the problem of finding asymptotic formulae for the sixth and eighth moments of the Riemann zeta function was first described by Conrey and Ghosh \cite{CGh} and Conrey and Gonek \cite{CG}. These ideas were refined by Ng \cite{Ng2}, showing that a certain smoothed version of the $k=\ell=3$ case of Conjecture \ref{manc} implies an asymptotic expansion for the sixth moment, and by Ng \emph{et al} \cite{Ngetal}, showing that a similar conjecture for the $k=\ell=4$ case of Conjecture \ref{manc} and the Riemann hypothesis together imply an asymptotic formula for the eighth moment. Moreover, in the series of papers \cite{CK1,CK2,CK3,CK4,CK5}, Conrey and Keating have investigated this connection for $k=\ell\in\mathbb{N}$ and provided a description of how Conjecture \ref{manc} leads to a conjecture for the $2k$th moments of the Riemann zeta function on the critical line. \\

The asymptotic order of $D_{h,k,{\ell}}(x)$ is fairly well understood. Regarding upper bounds, it follows from the general theorem of Nair and Tenenbaum \cite{NandT} that 
\begin{eqnarray}\label{boundsd}
D_{h,k,{\ell}}(x)=O_{h,k,{\ell}}(x\log^{k+{\ell}-2}x),
\end{eqnarray} 
with uniformity in the $h$ aspect following from the work of Henriot \cite{Hen}---the paper of Ng and Thom cited above addresses these matters in detail. However,  when $k,{\ell}\geq 3$, explicit bounds on the size of the constant implied in (\ref{boundsd}) have not yet appeared in the literature. Regarding lower bounds, the best general result in the literature is due to Ng and Thom \cite{NgT}, who showed that for $k,{\ell}\geq 3$  there is a $B_{k,{\ell}}>0$ such that for 
\begin{eqnarray}
h\leq \exp\left(B_{k,{\ell}}(\log x\log\log x)^{(\min(k,{\ell})-1)/(\min(k,{\ell})-1.99)}\right)\nonumber
\end{eqnarray} 
we have
\begin{eqnarray}\label{nandthombound}
\frac{D_{h,k,{\ell}}(x)}{x\log^{k+{\ell}-2}x}\geq \left(1+O_{k,{\ell}}\left(\frac{\log\log h}{\log x}\right)\right)\frac{2^{2-k-{\ell}}C_{k,{\ell}}f_{k,{\ell}}(h)}{(k-1)!({\ell}-1)!}.
\end{eqnarray} 
Moreover, regarding averages over $h$, Matomaki, Radziwill and Tao \cite{TMR} have recently shown that the conjectured asymptotic (\ref{lead}) holds for $k$, ${\ell}\geq 2$ and almost all $h\leq H$ provided that $x^{8/33+\epsilon}\leq H\leq x^{1-\epsilon}$, improving on previous work of Baier, Browning, Marasingha and Zhao \cite{BBMZ} on the case $k={\ell}=3$.\\

The additive divisor problem is closely related to the problem of improving the ``exponent of distribution'' for the generalised divisor problem in arithmetic progressions. An exponent of distribution is a lower bound on the lengths of arithmetic progressions $n\equiv h\pmod q$, $(h,q)=g$, in which $d_k(n)$ is ``well distributed'': 
\begin{definition}\label{def1}
A real number $0< \theta_{g,k}\leq 1$ is an exponent of distribution for $d_k(n)$ if for every fixed  $\epsilon>0$, $q\leq x^{\theta_{g,k}-\epsilon}$ and each residue class $h\not\equiv 0 \pmod q$,  $(h,q)=g$, we have
\begin{eqnarray}\label{sume}
\sum_{\substack{n\leq x\\n\equiv h \pmod q}} d_k(n)= \frac{1}{\phi\left(q/g\right)}
{\emph{Res}}\left(\frac{x^{s}}{s}\sum_{(n,q)=g}\frac{d_k(n)}{n^s},s=1     \right)
+O_{\epsilon,\delta,k}\left(\frac{x^{1-\delta}}{\phi\left(q/g\right)}\right)\nonumber\\
\end{eqnarray}
for some fixed $\delta>0$.
\end{definition}
A statement that is slightly stronger than (\ref{sume}) in terms of the dependence of $q$ and $g$ on $x$ is 
\begin{eqnarray}\label{sume2}
\sum_{\substack{n\leq x\\n\equiv h \pmod q}} d_k(n)= \frac{1}{\phi\left(q/g\right)} \sum_{n\leq x/g}\chi_0(n)d_k(gn)+O_{\epsilon,\delta,k}\left(\frac{x^{1-\delta}}{\phi\left(q/g\right)}\right)
\end{eqnarray}
where $\chi_0$ is the principal Dirichlet character to the modulus $q/g$. It is expected that $\theta_{g,k}=1$ for all $k$ provided that $g$ is not large, say $g\leq x^{1-\epsilon}$, but these are notoriously difficult problems for all $k\geq 2$. The strongest results in the literature are as follows. For $k=2$, Hooley \cite{H} established that  
we may take $\theta_{1,2}= 2/3$. We have $\theta_{g,3}= 21/41$ for all $g$ due to Heath-Brown \cite{H2}, $\theta_{1,4}=1/2$ due to Linnik \cite{Lin2}, and $\theta_{1,5}= 9/20$, $\theta_{1,6}= 5/12$ and $\theta_{1,k}= 8/3k$ for $k\geq 7$ due to Friedlander and Iwaniec \cite{FI}. For $k> 2$, the only known $k$ for which an exponent of distribution greater than $1/2$ is known is $k=3$, and both proofs (including the inferior exponent $58/115$ attributed to Friedlander and Iwaniec above) depend on Deligne's Riemann hypothesis for algebraic varieties over finite fields. For specific moduli, further increments have also been achieved. For instance, Fouvry, Kowalski and Michel \cite{FKM} have shown that essentially (\ref{sume}) holds for $k=3$ for all primes $q\leq x^{12/23}$. Moreover, Nguyen \cite{Ngy} has shown that a certain average exponent of distribution is sufficient to yield the main term in the $h=1$ and $k=\ell=3$ case of Conjecture \ref{manc}.

With the exception of Heath-Brown's
result for $k=3$, the above results on exponents of distribution are stated only for $g=1$, which is usually because $g=1$ is the only value that is required in applications to primes in arithmetic progressions. For applications to additive divisor sums, we require exponents of distribution for all $g$ in some range as $x\rightarrow\infty$. However, for each $k$ the distinction between the problems of establishing the existence of exponents of distribution $\theta_{g,k}$ is a relatively minor technicality and the outcome is essentially the same for all $g$. Under this proviso, we shall simply write $\theta_k=\theta_{g,k}$ because exponents of distribution appear as variables in this work.

\subsection{The minorants}\label{minorants}  The main idea in this work is that the (suitably normalised) minorants  
\begin{eqnarray} \label{ddefa}
d_{\ell}(n,A)=\sum_{q|n:q\leq n^A}d_{{\ell}-1}(q)\hspace{1cm}A\in(0,1].
\end{eqnarray} 
provide good approximations to $d_k(n)$ in arithmetic progressions. To explain what this means, we note that if $f:\mathbb{N}\rightarrow\mathbb{C}$, then 
\begin{eqnarray}\label{approxa}
\sum_{n\leq x}\left(\sum_{\substack{d|n\\d\leq n^A}}f(d)\right)=x\sum_{n\leq x^A}\frac{f(n)}{n}+O\left(\sum_{n\leq x^A}\frac{|f(n)|}{n^{1-1/A}}\right)\hspace{1cm}A\in(0,1],\nonumber\\
\end{eqnarray}
uniformly for $A\geq A_0>0$. Taking $f(n)=d_{k-1}(n)$ in (\ref{approxa}) it is elementary that $d_k(n,A)$
 approximates $A^{k-1}d_k(n)$ in the mean, that is
\begin{eqnarray}\label{diffb}
\sum_{n\leq x}d_k(n,A)=A^{k-1}\sum_{n\leq x}d_k(n)+O_A\left(x\log^{k-2}x\right).
\end{eqnarray}
The analogue of this principle for arithmetic progressions is 
\begin{eqnarray}\label{approxa2}
\sum_{\substack{n\leq x\\n\equiv h\pmod q}}\left(\sum_{\substack{d|n\\d\leq n^A}}f(d)\right)=\frac{x}{q}\sum_{\substack{n\leq x^A\\(n,q)|h}}\frac{(n,q)f(n)}{n}+O\left(\sum_{n\leq x^A}\frac{|f(n)|}{n^{1-1/A}}\right),
\end{eqnarray}
 which may be proved by interchanging the order of summation and trivially estimating the length of the resulting arithmetic progression. Yet, in applications to correlation problems, the error term in (\ref{approxa2}) is too weak; typically we need a factor of $1/q$, uniformly for $q\leq x^{C}$ as $x\rightarrow\infty$ for some $C>1-A$. \\

 Our first theorem (which is proved is Section \ref{arithproof}) refines (\ref{approxa2}) when $f(n)=d_{k-1}(n)$ with a suitable error term.

\begin{theorem}\label{arith}If $h,k\in\mathbb{N}$ and $\epsilon>0$ are fixed and $q\leq x^{\min(\theta_k,A\theta_{k-1})-\epsilon}$, then 
\begin{eqnarray}\label{mainpoly0}
\sum_{\substack{n\leq x\\n\equiv h \pmod q}} d_k(n,A)=A^{k-1}\sum_{\substack{n\leq x\\n\equiv h \pmod q}}d_k(n)+ O_{A,\epsilon,h,k}\left(\frac{x\log^{k-2}x}{q}\right).
\end{eqnarray}
In fact we prove the following 
\begin{eqnarray}
\sum_{\substack{n\leq x\\n\equiv h \pmod q}} d_k(n,A)=\frac{x}{q}\sum_{\substack{n\leq x^A\\(n,q)|h}}\frac{(n,q)d_{k-1}\left(n\right)}{n}+O_{A,\epsilon,h,k}\left(\frac{x\log^{k-2}x}{q}\right).
\end{eqnarray}\end{theorem}

Theorem \ref{mainlem} (which is proved in Section \ref{pott}) gives an asymptotic expansion for the correlation of $d_k(n)$ and $d_{\ell}(n,A)$.
\begin{theorem}\label{mainlem}If $A<\theta_{k}$ and $h,k,{\ell}\in\mathbb{N}$ are fixed, 
then there is a $\delta=\delta(k,{\ell})>0$ and a polynomial $P_{A,h,k,{\ell}}$ of degree $k+{\ell}-2$ such that  
\begin{eqnarray}\label{mainpoly2}
\sum_{n\leq x}d_k(n+h)d_{\ell}(n,A)=xP_{A,h,k,{\ell}}(\log x)+O_{A,h,k,{\ell}}\left(x^{1-\delta}\right).
\end{eqnarray}
An explicit formula for $P_{A,h,k,{\ell}}$ is given in (\ref{polydef}). In particular, the coefficient of the leading term is $A^{{\ell}-1}C_{k,{\ell}}f_{k,{\ell}}(h)/(k-1)!({\ell}-1)!$ where $C_{k,{\ell}}$ and $f_{k,{\ell}}(h)$ are defined in (\ref{c}) and (\ref{fform}). 
\end{theorem}
\noindent We note here that if $\theta_k>1/2$ and ${\ell}=2$, then $A=1/2$ is admissible in Theorem \ref{mainlem}, which thus yields an alternative proof of  (\ref{adv}) in such cases. For example, in Section \ref{append} we carry out this calculation in the case $k={\ell}=2$ to reproduce Estermann's asymptotic expansion (\ref{est}) explicitly. \\

Theorem \ref{anotherpoly} (which is proved in Section \ref{pots}) gives an asymptotic formula for the correlation of $d_k(n,A)$ with $d_{\ell}(n,B)$.

\begin{theorem}\label{anotherpoly} If $A\leq 1$, $B<\min(\theta_k,A\theta_{k-1})$ and $h,k,{\ell}\in\mathbb{N}$ are fixed, then 
\begin{eqnarray}\label{mainpoly1}
\sum_{n\leq x}d_k(n+h,A)d_{\ell}(n,B)&=&\frac{A^{k-1}B^{{\ell}-1}C_{k,{\ell}}f_{k,{\ell}}(h)}{(k-1)!({\ell}-1)!}x\log^{k+{\ell}-2} x\nonumber\\
&+&O_{A,B,h,k,{\ell}}\left(x \log^{k+{\ell}-3}\right)
\end{eqnarray}
\end{theorem}

\subsection{Applications to the additive divisor problem}\label{secres2} 
The first three results of this section follow immediately from the results of Section \ref{minorants}. Firstly, Corollary \ref{t1} sharpens the lower bound (\ref{nandthombound}) when $h$ is fixed and $k$ is sufficiently large in comparison with ${\ell}$, providing progress on the question raised in the concluding remarks 1 of \cite{NgT}.
\begin{corollary}\label{t1}
For fixed $h,k,{\ell}\in\mathbb{N}$ we have
\begin{eqnarray}\label{mybound2}
\liminf_{x\rightarrow\infty}\frac{D_{h,k,{\ell}}(x)}{x\log^{k+{\ell}-2}x}\geq \theta_k^{{\ell}-1}\frac{C_{k,{\ell}}f_{k,{\ell}}(h)}{(k-1)!({\ell}-1)!}.
\end{eqnarray}
\end{corollary}
\begin{proof}
This is immediate from  Theorem \ref{mainlem} because $d_{\ell}(n)\geq d_{\ell}(n,A)$.
\end{proof}
\noindent 
For instance, given Heath-Brown's exponent $\theta_3=21/41$, it follows from Corollary \ref{t1} that
\begin{eqnarray}\label{mybound3}
\liminf_{x\rightarrow\infty}\frac{D_{h,3,3}(x)}{x\log^{4}x}\geq 0.262\frac{C_{3,3}f_{3,3}(h)}{4}. \nonumber
\end{eqnarray}

Corollary \ref{t2} gives an equivalent condition for the conjectured asymptotic (\ref{lead}). 
\begin{corollary}\label{t2} For fixed $h,k,{\ell}\in\mathbb{N}$, the asymptotic (\ref{lead}) holds if and only if 
\begin{eqnarray}\label{diff}
\sum_{n\leq x}d_k(n+h)\left(d_{\ell}(n)-B^{1-l}d_{\ell}(n,B)\right)=o\left( x\log^{k+{\ell}-2} x    \right)
\end{eqnarray}
for some (and therefore every) $B< \theta_k$.
\end{corollary}
\begin{proof}
 Compare (\ref{lead}) with Theorem \ref{mainlem} and note that the latter result holds for all $B<\theta_k$, while (\ref{lead}) does not depend on $B$ at all.
\end{proof}

In support of the plausibility of (\ref{diff}), we note that  
\begin{corollary}\label{t6} If $A<\theta_{\ell}$, $B<\min(\theta_k,A\theta_{k-1})$ and  $h,k,{\ell}\in\mathbb{N}$ are fixed, then 
\begin{eqnarray}
\sum_{n\leq x}d_k(n+h,A)\left(d_{\ell}(n)-B^{1-{\ell}}d_{\ell}(n,B)\right)=O_{A,B,h,k,{\ell}}\left( x\log^{k+{\ell}-3} x    \right).
\end{eqnarray}
\end{corollary}
\begin{proof}
This follows from Theorem \ref{anotherpoly} by swapping variables $A,B$ and $k,{\ell}$.
\end{proof}

Our last result is a Tauberian theorem (which is proved in Section \ref{pot3}). This is analogous to the relationship between the Prime Number Theorem and the non-vanishing of $\zeta(1+it)$. 
\begin{theorem}\label{t3}Let $h,k,{\ell}\in\mathbb{N}$ and $0\leq y<\infty$ be fixed, then the function 
\begin{eqnarray}
\mathcal{D}_{h,k,{\ell}}(s,y)=\sum_{n=1}^{\infty}\frac{d_k(n+h)d_{\ell}\left(n,\frac{y}{\log n}\right)}{(n+h)^s}\hspace{1cm}(\sigma>1) \nonumber
\end{eqnarray}
has an analytic continuation to the complex plane except for a pole of order $k-1$ at $s=1$ and, if the limit
\begin{eqnarray}
\lim_{y\rightarrow\infty}\mathcal{D}_{h,k,{\ell}}(1+it,y) \nonumber
\end{eqnarray} is continuous for $t\neq 0$, then we have
\begin{eqnarray}\label{asymor}
\frac{D_{h,k,{\ell}}(x)}{x\log^{k+{\ell}-2}x}\sim\frac{C_{k,{\ell}}f_{k,{\ell}}(h)}{(k-1)!({\ell}-1)!} \nonumber
\end{eqnarray}
as $x\rightarrow\infty$.
\end{theorem}


\section{Some definitions}\label{coeffs}
\noindent Definitions \ref{ftg}---\ref{forma} will be needed in the course of the proofs.
\begin{definition} \label{ftg} For $j,n\in\mathbb{N}$ we define   
\begin{eqnarray}
a_n(j)=\frac{d^n}{ds^n}(s-1)^j\zeta^j(s)\biggr\vert _{s=1} \nonumber
\end{eqnarray}
and
\begin{eqnarray}\label{fdef}
c_n(j)=\frac{1}{n!}\frac{d^n}{ds^n}\frac{(s-1)^j\zeta^j(s)}{s}\biggr\vert _{s=1}=\sum_{r=0}^{n}\frac{(-1)^{n-r}a_r(j)}{r!}. \nonumber
\end{eqnarray}
\end{definition}
\begin{definition}\label{fth}
 For $h,k,{\ell}\in\mathbb{N}$ and $h=\prod p^{\gamma}$, we also define
\begin{eqnarray}
C_{k,{\ell}}(s,w)= \prod_p \left((1-p^{-w-1})^{{\ell}-1}+\frac{(1-p^{-s})^k} {1-p^{-1}} -    \frac{(1-p^{-s})^k(1-p^{-w-1})^{{\ell}-1}}{1-p^{-1}}\right)\nonumber
\end{eqnarray}
for complex numbers $s,w$ such that $\sigma+\Re w>0,k\sigma+\Re w>0,\sigma+(l-1)\Re w>0$ and $k\sigma+(l-1)\Re w>0$,
\begin{eqnarray}\label{fdef}
f_{h,k,{\ell}}(s,w)=\prod_{p|h} \frac{(1-p^{-1})\sum_{\alpha=0}^{\gamma}d_{{\ell}-1}(p^{\alpha})  
\sum_{\beta=\alpha}^{\infty}d_k(p^{\beta})p^{-\beta s-\alpha w}+d_k(p^{\gamma})\sum_{\alpha=\gamma+1}^{\infty}d_{{\ell}-1}(p^{\alpha}) p^{-\alpha (w+1)}}{(1-p^{-1})(1-p^{-s})^{-k}+(1-p^{-w-1})^{1-{\ell}}    -1} ,\nonumber
\end{eqnarray}
and
\begin{eqnarray}\label{dr}
C_{k,{\ell}}(s,w)f_{h,k,{\ell}}(s,w)=\sum_{q=1}^{\infty}\frac{\varphi_{h,k,{\ell}}(q,s)}{q^{w}}.
\end{eqnarray}
\end{definition}

\begin{definition}\label{form}
For $m<k$ and $n<{\ell}$ we define
\begin{eqnarray}\label{coeffse}
b_{h,k,{\ell},m,n}=\sum_{i=0}^{k-1-m}\sum_{j=0}^{{\ell}-1-n}\frac{a_{{\ell}-1-n-j}(l-1)c_{k-1-m-i}(k)}{({\ell}-1-n-j)!}
\frac{\partial^{i}}{i!\partial s^{i}}\frac{\partial ^{j}}{j!\partial w^{j}}\sum_{q=1}^{\infty}\frac{\varphi_{h,k,{\ell}}(q,s)}{q^{w}}\biggr\vert _{w=0,s=1} \nonumber \\
\end{eqnarray}
and note that the Dirichlet series in (\ref{coeffse}) converge absolutely. In particular, we have $b_{h,k,{\ell},k-1,{\ell}-1}=C_{k,{\ell}}f_{k,{\ell}}(h)$
where $C_{k,{\ell}}$ and $f_{k,{\ell}}(h)$ are defined in  (\ref{c}) and (\ref{fform}).
\end{definition}
\begin{definition}\label{forma}
Lastly, for $m<k+{\ell}-2$ and $0<A\leq 1$ we define
\begin{eqnarray}
a_{A,h,k,{\ell},m}&=&(-1)^{m}\sum_{j=m-{\ell}+2}^{k-1}{i\choose j}\sum_{r=m-{\ell}+2}^{j}\sum_{v=0}^{r-m+{\ell}-2}\frac{(-A)^{r-j-v+{\ell}-1}a_v({\ell}-1)(v-{\ell}+1)_r}{v!}\nonumber\\&\times &{{\ell}-v-2\choose j-r}
 \sum_{i=j}^{k-1}\frac{c_{k-1-i}(k)}{i!} \frac{\partial^{i-j}}{\partial s^{i-j}} \frac{\partial^{j+{\ell}-r-2}}{\partial w^{j+{\ell}-r-2}} \sum_{1}^{\infty}\frac{\varphi_{h,k,{\ell}}^{}(d,s)}{d^w}\biggr\vert _{w=0,s=1}.
\end{eqnarray}
\end{definition}

\section{Proof of Theorem \ref{arith}}\label{arithproof}
We begin by writing 
\begin{eqnarray}
d_k(n,A)&=&\left(d_k(n)-\sum_{\substack{d|n\\d\leq x^{1-A}}}d_{k-1}\left(\frac{n}{d}\right)-\sum_{\substack{d|n\\d> x^{1-A}}}d_{k-1}\left(\frac{n}{d}\right)\right)+d_k(n,A)
\nonumber\\&=&d_k(n)-\sum_{\substack{d|n\\d<x^{1-A}}}d_{k-1}\left(\frac{n}{d}\right)+\sum_{\substack{d|n\\nx^{A-1}\leq d\leq n^A}}d_{k-1}(d),
\end{eqnarray}
so that 
\begin{eqnarray}\label{prof}
\sum_{\substack{n\equiv h\pmod q\\n\leq x}}d_k(n,A) &=& \sum_{\substack{n\equiv h\pmod q\\n\leq x}}d_{k}(n)
-\sum_{d<x^{1-A}} \sum_{\substack{dn\equiv h\pmod q\\n\leq x/d}} d_{k-1} \left(n\right)\nonumber\\
&+&   \sum_{d\leq x^A}d_{k-1}(d)  \sum_{\substack{dn\equiv h\pmod q\\d^{\frac{1-A}{A}}\leq n\leq x^{1-A}}}   1.
\end{eqnarray}
Firstly, using Definition \ref{def1} in the form (\ref{sume2}), the first term on the right hand side of (\ref{prof}) is 
\begin{eqnarray}
\frac{1}{\phi\left(q/(h,q)\right)} \sum_{n\leq x/(h,q)}\chi_0(n)d_k((h,q)n)+O_{\epsilon,\delta,k}\left(\frac{x^{1-\delta}}{\phi\left(q/(h,q)\right)}\right)
\end{eqnarray}
 for $q\leq x^{\theta_k-\epsilon}$, where $\chi_0$ is the principal character to the modulus $q/(q,h)$. Secondly, the third term on the right hand side of (\ref{prof}) is 
\begin{eqnarray}\label{exa}
&\leq &  \sum_{n< x^{1-A}}\sum_{\substack{dn\equiv h\pmod q\\ d\leq x^{A}}}d_{k-1}(d)\nonumber\\
&\leq &\sum_{\substack{n< x^{1-A}\\(n,q)|h}}\sum_{\substack{d\equiv \overline{(n/(n,q))}(h/(n,q))\pmod {q/(n,q)}\\ d\leq x^{A}}}d_{k-1}(d).
\end{eqnarray}
Using (\ref{sume2}) with $q\leq x^{A\theta_{k-1}-\epsilon}$ and 
\begin{eqnarray}
g=\left(\overline{\left(\frac{n}{(n,q)}\right)}\left(    \frac{h}{(n,q)}\right),      \frac{q}{(n,q)}\right)=\frac{(h,q)}{(n,q)}
\nonumber
\end{eqnarray} 
the inner summation in (\ref{exa}) is
\begin{eqnarray}
&\ll_{A,\epsilon}&\frac{1}{\phi\left(\frac{q}{(h,q)}\right)} \sum_{m\leq x^A/g}\chi_0(m)d_{k-1}(gm)\nonumber\\
&=& \frac{1}{\phi\left(\frac{q}{(h,q)}\right)}\sum_{d\leq x^A}d_{k-2}(d)\chi_0(d)\sum_{m\leq x^A/dg}\chi_0(m)\nonumber\\
&\ll_{A,\epsilon}& \frac{x^A\prod_{p|\frac{q}{(h,q)}}\left(1-\frac{1}{p}\right)}{g\phi\left(\frac{q}{(h,q)}\right)}\sum_{d\leq x^A}\frac{d_{k-2}(d)}{d}\nonumber\\
&\ll_{A,\epsilon}&\frac{(n,q)}{q}x^A\log^{k-2}x\nonumber
\end{eqnarray}
so (\ref{exa}) is $\ll_{A,\epsilon,h} x(\log x)^{k-2}/q$ because $(n,q)|h$. As such, (\ref{prof}) is  
\begin{eqnarray}\label{prof2}
&=&\frac{1}{\phi\left(q/(h,q)\right)} \sum_{n\leq x/(h,q)}\chi_0(n)d_k((h,q)n)
\nonumber\\&-&\sum_{d<x^{1-A}} \sum_{\substack{dn\equiv h\pmod q\\n\leq x/d}} d_{k-1} \left(n\right)+   O_{A,\epsilon,h,k}\left(\frac{x\log^{k-2}x}{q}\right)\nonumber\\
&=&\frac{1}{\phi\left(q/(h,q)\right)} \sum_{n\leq x/(h,q)}\chi_0(n)d_k((h,q)n)\nonumber\\&-&\frac{1}{\phi\left(q/(h,q)\right)}\sum_{\substack{d<x^{1-A}\\(d,q)|h}}\sum_{n\leq (x/(h,q))/(d/(d,q))}\chi_0(n)d_{k-1}\left(\frac{(h,q)n}{(d,q)}\right)+O_{A,\epsilon,h,k}\left(\frac{x\log^{k-2}x}{q}\right)\nonumber\\
\end{eqnarray}
 for $q\leq x^{\min(\theta_k,A\theta_{k-1})-\epsilon}$, where we have used Definition \ref{def1} again to write the second sum on the right hand side of (\ref{prof2}) in terms of the principal character. We now write the second sum as 
 \begin{eqnarray}\label{prof3}
&-&\sum_{m\leq x/(h,q)}\sum_{\substack{d<x^{1-A}\\  (d,q)|h \\  d/(d,q)|m }}\chi_0\left(\frac{(d,q)m}{d} \right)d_{k-1}\left(\frac{(h,q)m}{d}\right)\nonumber\\
\end{eqnarray}
and note that, since $(d/(d,q),q/(h,q))=1$, the character $\chi_0((d,q)m/d)$ is non-zero precisely when $(m,q/(h,q))=1$ so we may replace it by $\chi_0(m)$. Thus, extending the sum over all $d$, (\ref{prof3}) is
\begin{eqnarray}
&=&-\sum_{m\leq x/(h,q)}\chi_0 \left(m\right)\sum_{\substack{ d/(d,q)|m\\(d,q)|h }}d_{k-1}\left(\frac{(h,q)m}{d}\right)
\nonumber
\end{eqnarray}
\begin{eqnarray}\label{proo6}
&+&\sum_{m\leq x/(h,q)}\chi_0 \left(m\right)\sum_{\substack{x^{1-A}\leq d\leq x/(h,q)\\  d/(d,q)|m \\(d,q)|h}}d_{k-1}\left(\frac{(h,q)m}{d}\right)\nonumber\\
&=&-\sum_{m\leq x/(h,q)}\chi_0 \left(m\right)d_{k}\left((h,q)m\right)
\nonumber\\&+&\sum_{m\leq x/(h,q)}\chi_0 \left(m\right)\sum_{\substack{x^{1-A}\leq d\leq x/(h,q)\\  d/(d,q)|m\\(d,q)|h }}d_{k-1}\left(\frac{(h,q)m}{d}\right),
\end{eqnarray}
where in the first sum above we have used the identity 
\begin{eqnarray}
d_{k}\left((h,q)m\right)=\sum_{\substack{ d/(d,q)|m\\(d,q)|h }}d_{k-1}\left(\frac{(h,q)m}{d}\right).\nonumber
\end{eqnarray}
By (\ref{prof2}) and (\ref{proo6}), (\ref{prof}) is
\begin{eqnarray}\label{prof4}
\frac{1}{\phi\left(q/(h,q)\right)}\sum_{m\leq x/(h,q)}\chi_0 \left(m\right)\sum_{\substack{x^{1-A}\leq d\leq x/(h,q)\\  d/(d,q)|m\\(d,q)|h }}d_{k-1}\left(\frac{(h,q)m}{d}\right)+O_{A,\epsilon,h,k}\left(\frac{x\log^{k-2}x}{q}\right).\nonumber\\
\end{eqnarray}
To evaluate the main term in (\ref{prof4}), we write this as
\begin{eqnarray}\label{abovew}
&=&\frac{1}{\phi\left(q/(h,q)\right)}\sum_{\substack{x^{1-A}\leq d\leq x/(h,q)\\  (d,q)|h }}  \sum_{m\leq (x/(h,q))/(d/(d,q))}\chi_0 \left(\frac{dm}{(d,q)}\right)   d_{k-1}\left(\frac{(h,q)m}{(d,q)}\right)\nonumber\\
&=&\frac{1}{\phi\left(q/(h,q)\right)}\sum_{\substack{x^{1-A}\leq d\leq x/(h,q)\\}} \sum_{\substack{n\leq x/d\\(d,q)|h\\((h,q)/(d,q))|n}}\chi_0 \left(\frac{dn}{(h,q)}\right)   d_{k-1}\left(n\right)\nonumber\\
&=&\frac{1}{\phi\left(q/(h,q)\right)}\sum_{\substack{n\leq x^{A}\\}}  \chi_0 \left(n\right) d_{k-1}\left(n\right)\sum_{\substack{x^{1-A}\leq d\leq x/(h,q)\\d\leq x/n\\(d,q)|h \\((h,q)/(d,q))|n}}\chi_0 \left(d\right).
\end{eqnarray}
Now the condition $(d,q)|h$  is automatic when $\chi_0 \left(d\right)$ is non zero. Since this implies that $(d,q)=(d,(h,q))$ and thus $(d,q)|(h,q)$, the conditions $(d,q)|h$ and $((h,q)/(d,q))|n$ may be replaced by the single condition $(h,q)|n$ so the above is simply 
\begin{eqnarray}
&&\frac{1}{\phi\left(q/(h,q)\right)}\sum_{\substack{n\leq x^A/(h,q)}}\chi_0 \left(n\right)   d_{k-1}\left((h,q)n\right)\sum_{\substack{x^{1-A}\leq d\leq x/n(h,q) }}\chi_0(d) \nonumber\\
&=&\frac{1}{\phi\left(q/(h,q)\right)}\sum_{\substack{n\leq x^A}/(h,q)}\chi_0 \left(n\right)   d_{k-1}\left((h,q)n\right)\sum_{\substack{d\leq x/n(h,q)}}\chi_0(d)+O_{A,h,k}\left(\frac{x\log^{k-2}x}{q}\right) \nonumber\\
&=&\frac{x}{q}\sum_{\substack{n\leq x^A/(h,q)}}\frac{\chi_0 \left(n\right)   d_{k-1}\left((h,q)n\right)}{n}+O_{A,\epsilon,h,k}\left(\frac{x\log^{k-2}x}{q}\right). \nonumber
\end{eqnarray}
Lastly, since $(n/(h,q),q/(h,q))=1$ is equivalent to $(n,q)|h$, this is
\begin{eqnarray}
\frac{x}{q}\sum_{\substack{n\leq x^A\\(n,q)|h}}\frac{(n,q)d_{k-1}\left(n\right)}{n}+O_{A,\epsilon,h,k}\left(\frac{x\log^{k-2}x}{q}\right)\nonumber
\end{eqnarray}
and, by partial summation, the remainder of the proof is straightforward.

\section{Proof of Theorem \ref{mainlem}}\label{pott}
To begin, we have 
\begin{eqnarray}\label{p1}
\sum_{n\leq x}d_k(n+h)d_{\ell}(n,A)&=&\sum_{n\leq x}d_k(n+h)\sum_{\substack{q|n\\q\leq n^A}}d_{{\ell}-1}(q)\nonumber\\
&=&\sum_{q\leq x^A}d_{{\ell}-1}(q)\sum_{\substack{q^{1/A}+h\leq n\leq x+h\\n\equiv h \pmod q}}d_k(n).
\end{eqnarray}\\

\noindent Using Definition \ref{def1} to evaluate the inner summations on the right hand side of (\ref{p1}), for $0<A<\theta_{k}$ we see that (\ref{p1}) is 
\begin{eqnarray}\label{p2}
\sum_{n\leq x}d_k(n+h)d_{\ell}(n,A)
&=&\sum_{q\leq x^A}\frac{d_{{\ell}-1}(q)}{\phi\left(\frac{q}{(h,q)}\right)}
\textrm{Res}\left(\frac{(x+h)^{s}}{s}\sum_{(n,q)=(h,q)}\frac{d_k(n)}{n^s},s=1\right)
\nonumber
\end{eqnarray}
\begin{eqnarray}
&-& \sum_{q\leq x^A}\frac{d_{{\ell}-1}(q)}{\phi\left(\frac{q}{(h,q)}\right)}
\textrm{Res}\left(\frac{(q^{1/A}+h-\delta_A(q))^{s}}{s}\sum_{(n,q)=(h,q)}\frac{d_k(n)}{n^s},s=1\right)\\
&+&O_{A,\epsilon,\delta,k}\left(x^{1-\delta}\sum_{q\leq x^A}\frac{d_{{\ell}-1}(q)}{\phi\left(\frac{q}{(h,q)}\right)}\right)\nonumber\\&+&O_{A,\epsilon,\delta,k}\left(x^{1-\delta}\sum_{q\leq x^A}\frac{d_{{\ell}-1}(q)(q^{1/A}+h-\delta_A(q))^{1-\delta}}{\phi\left(\frac{q}{(h,q)}\right)}\right),\nonumber
\end{eqnarray}

\noindent  where $\delta_A(q)=0$ or $1$ depending on whether $q^{1/A}$ is an integer or not. The summations in the error terms in the third and fourth lines of (\ref{p2}) are $O_{A,h}(\log^{{\ell}-1} x)$, so it remains to evaluate the first two terms. It is convenient to evaluate these terms in different ways due to the different ways in which each depends on $A$ and therefore $q$, which we shall do in Sections \ref{evpt} and \ref{evst} below.

\subsection{The primary term}\label{evpt}
We begin by evaluating the first term on the right hand side of (\ref{p2}).  Let $\chi_0$ denote the principal character to the modulus $q/g$, where $q=\prod p^{\alpha}$, $h=\prod p^{\gamma}$ and $g=(h,q)=\prod p^{\delta}$ so $\delta=\min(\alpha,\gamma)$. We have

\begin{eqnarray}\label{g}
\sum_{(n,q)=g}\frac{d_k(n)}{n^s}=\sum_{n=1}^{\infty}\frac{\chi_0(n)d_k(gn)}{(gn)^s}&=&\prod_p\sum_{\beta=0}^{\infty}d_k(p^{\beta+\delta})\chi_0(p^{\beta})p^{-(\beta+\delta)s}
\nonumber\\&=&L^k(s,\chi_0)b_{h,k}(s,q) \nonumber
\end{eqnarray}
where

\begin{eqnarray}\label{adef}
b_{h,k}(s,q)=\prod_{p|g}(1-\chi_0(p)p^{-s})^k\sum_{\beta=\delta}^{\infty}d_k(p^{\beta})\chi_0(p^{\beta-\delta})p^{-\beta s} \nonumber
\end{eqnarray}
is a multiplicative function of $g$ for all $k,s$. By Cauchy's theorem, the first term on the right hand side of (\ref{p2}) is

\begin{eqnarray}\label{mainres}
&=&\frac{1}{(k-1)!}\frac{\partial^{k-1}}{\partial s^{k-1}}\frac{(s-1)^k\zeta^k(s)Z_{h,k,{\ell}}\left(s,x^A\right)(x+h)^s}{s}\biggr\vert _{s=1}\nonumber\\
&=&\frac{1}{(k-1)!}\sum_{i=0}^{k-1}{k-1\choose i}\frac{\partial^{i}}{\partial s^{i}}Z_{h,k,{\ell}}\left(s,x^A\right)\biggr\vert _{s=1}\frac{\partial^{k-1-i}}{\partial s^{k-1-i}}\frac{(s-1)^k\zeta^k(s)(x+h)^s}{s}\biggr\vert _{s=1}\nonumber\\
\end{eqnarray}
where

\begin{eqnarray}\label{sum}
Z_{h,k,{\ell}}\left(s,x^A\right)=\sum_{q\leq x^A}\frac{d_{{\ell}-1}(q) }{\phi \left(\frac{q}{(h,q)}\right)}\prod_{p|\frac{q}{(h,q)}}\left(1-p^{-s}\right)^kb_{h,k}(s,q),
\end{eqnarray}
and so our first task is to find asymptotic formulae for $\frac{\partial^i}{\partial s^i}Z_{h,k,{\ell}}\left(s,Q\right)$ at $s=1$ as $Q\rightarrow\infty$. To proceed, we note that the factor 

\begin{eqnarray}
\frac{d_{{\ell}-1}(q) }{\phi \left(\frac{q}{(q,h)}\right)}\prod_{p|\frac{q}{(q,h)}}\left(1-p^{-s}\right)^k \nonumber
\end{eqnarray}
of the summand in (\ref{sum}) is a multiplicative function of $q$ for all $h,k,s$, and we shall now show that $b_{h,k}(s,q)$ is also. From the definition of $b_{h,k}(s,q)$ above, we have 

\begin{eqnarray}\label{nine}
b_{h,k}(s,q)&=&\prod_{\substack{p|(q,h)\\p|\frac{q}{(q,h)}}}\frac{d_k(p^{\delta(q)})}{p^{\delta(q)s}} \prod_{\substack{p|(q,h)\\p\nmid\frac{q}{(q,h)}}} 
(1-p^{-s})^k\sum_{\beta=\delta(q)}^{\infty}d_k(p^{\beta})p^{-\beta s}.
\end{eqnarray}
If $q=rt$ with $(r,t)=1$, we have $(rt,h)=(r,h)(t,h)$ and $\delta(rt)=\delta(r)+\delta(t)$ for every $p$ where $\delta(n)$ is the $p$-adic valuation of $(n,q)$, so the inclusion $p|(rt,h)$ in (\ref{nine}) is multiplicative, that is

\begin{eqnarray}\label{ff}
b_{h,k}(s,rt)&=&\prod_{\substack{p|(r,h)\\p|\frac{r}{ (r,h)} \frac{t}{(t,h)}   }}\frac{d_k(p^{\delta(r)})}{p^{\delta(r)s}}\prod_{\substack{ p|(r,h)\\  p\nmid\frac{r}{ (r,h)} \frac{t}{(t,h)}}}(1-p^{-s})^k\sum_{\beta=\delta(r)}^{\infty}d_k(p^{\beta})p^{-\beta s}
\nonumber\\
&\times&\prod_{\substack{p|(t,h)\\p|\frac{r}{ (r,h)} \frac{t}{(t,h)}}}\frac{d_k(p^{\delta(t)})}{p^{\delta(t)s}}\prod_{\substack{ p|(t,h)\\  p\nmid \frac{r}{ (r,h)} \frac{t}{(t,h)}}}(1-p^{-s})^k\sum_{\beta=\delta(t)}^{\infty}d_k(p^{\beta})p^{-\beta s}.
\end{eqnarray}
Since $p|r$ implies $p\nmid t$, the intersection of the sets $p|(r,h)$ and $p| t/(t,h)$ in (\ref{ff}) is already empty. It follows that if $p|(r,h)$ then the inclusion $p| t/(t,h)$ and exclusion $p\nmid t/(t,h)$ is superfluous, and vice-versa. Therefore

\begin{eqnarray}\label{ffs}
b_{h,k}(s,rt)&=&\prod_{\substack{p|(r,h)\\p|\frac{r}{ (r,h)  }}}\frac{d_k(p^{\delta(r)})}{p^{\delta(r)s}}\prod_{\substack{ p|(r,h)\\  p\nmid\frac{r}{ (r,h)   }}}(1-p^{-s})^k\sum_{\beta=\delta(r)}^{\infty}d_k(p^{\beta})p^{-\beta s}
\nonumber\\
&\times&\prod_{\substack{p|(t,h)\\p|\frac{t}{  (t,h)   }}}\frac{d_k(p^{\delta(t)})}{p^{\delta(rt)s}}\prod_{\substack{ p|(t,h)\\  p\nmid \frac{t}{(t,h)   }}}(1-p^{-s})^k\sum_{\beta=\delta(t)}^{\infty}d_k(p^{\beta})p^{-\beta s} \nonumber \\
&=&b_{h,k}(s,r)b_{h,k}(s,t). \nonumber
\end{eqnarray}
Thus

\begin{eqnarray}\label{sum2}
Z_{h,k,{\ell}}\left(s,x^A\right)=\sum_{q\leq x^A}\phi_{h,k,{\ell}}(s,q) \nonumber
\end{eqnarray}
where the summand
\begin{eqnarray}
\phi_{h,k,{\ell}}(s,q)=\frac{d_{{\ell}-1}(q) }{\phi \left(\frac{q}{(q,h)}\right)}\prod_{p|\frac{q}{(q,h)}}\left(1-p^{-s}\right)^k\prod_{\substack{p|(q,h)\\p|\frac{q}{(q,h)}}}\frac{d_k(p^{\delta(q)})}{p^{\delta(q)s}} \prod_{\substack{p|(q,h)\\p\nmid\frac{q}{(q,h)}}} 
(1-p^{-s})^k\sum_{\beta=\delta(q)}^{\infty}d_k(p^{\beta})p^{-\beta s}\nonumber
\end{eqnarray}
is a multiplicative function of $q$ for all $h,k,{\ell},s$. \\

\noindent Since $\phi_{h,k,{\ell}}(s,q)$ is multiplicative, we define the Euler product 
\begin{eqnarray}\label{prod}
\Phi_{h,k,{\ell}}(s,w)=\prod_p \sum_{\alpha=0}^{\infty}\phi_{h,k,{\ell}}(s,p^{\alpha})p^{-\alpha w} \nonumber 
\end{eqnarray}
for values of $w\in\mathbb{C}$ for which the right hand side converges absolutely. If $p\nmid h$ then $\phi_{h,k,{\ell}}(s,p^{\alpha})=\phi_{1,k,{\ell}}(s,p^{\alpha})$ which gives
\begin{eqnarray}\label{fac}
\Phi_{h,k,{\ell}}(s,w)=
C_{k,{\ell}}(s,w)f_{h,k,{\ell}}(s,w)\zeta^{{\ell}-1}(w+1),
\end{eqnarray}
where $f_{h,k,{\ell}}(s,w)$ and $C_{k,{\ell}}(s,w)$
are defined in (\ref{fdef}). It follows that for fixed $h,i,k,{\ell}$ the Dirichlet series 
\begin{eqnarray}\label{ser}
\frac{\partial^i}{\partial s^i}C_{k,{\ell}}(s,w)f_{h,k,{\ell}}(s,w)=\sum_{q=1}^{\infty}\frac{\partial^i}{\partial s^i}\frac{\varphi_{h,k,{\ell}}(q,s)}{q^{w}} \nonumber
\end{eqnarray}
is absolutely convergent and bounded for such values of $s,w$. Thus, using the relation
\begin{eqnarray}\label{conv}
\phi_{h,k,{\ell}}(s,q)=\sum_{d|q}\varphi_{h,k,{\ell}}(d,s)\frac{d_{{\ell}-1}(q/d)}{q/d},
\end{eqnarray}
we have
\begin{eqnarray}\label{lop}
\frac{\partial^i}{\partial s^i}Z_{h,k,{\ell}}\left(s,Q\right)&=&\frac{\partial^i}{\partial s^i}\sum_{q\leq Q}\sum_{d|q}\varphi_{h,k,{\ell}}(d,s)\frac{d_{{\ell}-1}(q/d)}{q/d}\nonumber\\
&=&\frac{\partial^i}{\partial s^i}\sum_{d\leq Q}\varphi_{h,k,{\ell}}(d,s)\sum_{q\leq Q/d}\frac{d_{{\ell}-1}(q)}{q}\nonumber\\
&=&\frac{\partial^i}{\partial s^i}\sum_{d\leq Q}\varphi_{h,k,{\ell}}(d,s)\Bigg(\sum_{j=0}^{{\ell}-1}\frac{a_{{\ell}-1-j}({\ell}-1)\log^j(Q/d)}{({\ell}-1-j)!j!}\\
&+&\frac{1}{2\pi i}\int_{\left(\epsilon-2/l\right)}\zeta^{{\ell}-1}(w+1)\frac{\left(Q/d\right)^wdw}{w}\Bigg)\nonumber\\
&=&\sum_{0}^{{\ell}-1}\frac{a_{{\ell}-1-j}({\ell}-1)}{j!({\ell}-1-j)!}\frac{\partial^i}{\partial s^i}\sum_{d\leq Q}\varphi_{h,k,{\ell}}(d,s)\log^j(Q/d)\nonumber\\
&+&
O\left(Q^{\epsilon-2/{\ell}}\sum_{d\leq Q}   \left| \frac{\partial^i}{\partial s^i} \varphi_{h,k,{\ell}}(d,s)  \right| d^{2/{\ell}-\epsilon}  \right),
\nonumber
\end{eqnarray}
which follows from classical results on the error term in the generalised Dirichlet divisor problem (see Titchmarsh \cite{Titch}, Section 12). Expanding $\log^j(Q/d)$ as a polynomial in $\log Q$, (\ref{lop}) is

\begin{eqnarray}\label{A}
&=&\sum_{j=0}^{{\ell}-1}\frac{a_{{\ell}-1-j}({\ell}-1)}{j!({\ell}-1-j)!}\sum_{n=0}^{j}{j\choose n}\log ^nQ\frac{\partial^i}{\partial s^i}\sum_{d\leq Q}\varphi_{h,k,{\ell}}(d,s)(-\log d)^{j-n}
+O\left(Q^{\epsilon-2/{\ell}}\right)
\nonumber\\
&=&\sum_{n=0}^{{\ell}-1}\frac{\log ^nQ}{n!}\sum_{j=0}^{{\ell}-1-n}\frac{a_{{\ell}-1-n-j}({\ell}-1)}{j!({\ell}-1-n-j)!}\frac{\partial^i}{\partial s^i}\frac{\partial^{j}}{\partial w^{j}}\sum_{d\leq Q}\frac{\varphi_{h,k,{\ell}}(d,s)}{d^{w}}\biggr\vert _{w=0}+O\left(Q^{\epsilon-2/{\ell}}\right).\nonumber\\
\end{eqnarray}\\

\noindent We also have 
\begin{eqnarray}\label{zalc}
&&\frac{\partial^{k-1-i}}{\partial s^{k-1-i}}\frac{(s-1)^k\zeta^k(s)(x+h)^s}{s}\biggr\vert _{s=1}\nonumber\\&=&
(x+h)(k-1-i)!\sum_{r=0}^{k-1-i}\frac{a_r(k)}{r!}\sum_{m=0}^{k-1-i-r}\frac{(-1)^{k-1-i-r-m}\log^m(x+h)}{m!}\nonumber\\
&=&(x+h)(k-1-i)!\sum_{m=0}^{k-1-i}\frac{\log^{k-1-i-m}(x+h)}{(k-1-i-m)!}c_m(k)
\end{eqnarray}
and setting $Q=x^A$ in (\ref{A}) and using (\ref{zalc}), we conclude that (\ref{mainres}) is 
\begin{eqnarray}\label{ftl}
&=&(x+h)\sum_{m=0}^{k-1}\sum_{n=0}^{{\ell}-1}\frac {A^nb_{h,k,{\ell},m,n}}{m!n!}\log^{m}(x+h)\log^{n}x\nonumber\\
&+&O\left((x+h)x^{\epsilon-2A/{\ell}}\sum_{i=0}^{k-1}\sum_{m=i}^{k-1}\frac{\log^{k-1-m}(x+h)|c_{m-i}(k)|}{(k-1-m)!}
    \sum_{d\leq x^A}   \left| \frac{\partial^i}{\partial s^i} \varphi_{h,k,{\ell}}(d,s)  \right|_{s=1} d^{2/{\ell}-\epsilon} \right),\nonumber\\ 
&=&x\sum_{m=0}^{k-1}\sum_{n=0}^{{\ell}-1}\frac {A^nb_{h,k,{\ell},m,n}}{m!n!}\log^{m+n}x+O_{h,k,{\ell}}\left(x^{1-2A/{\ell}+\epsilon} \right)\nonumber\\
\end{eqnarray}
where the coefficients $b_{h,k,{\ell},m,n}$ are defined in Section \ref{coeffs}.

\subsection{The secondary term}\label{evst}
We now evaluate the second term on the right hand side of (\ref{p2}). By Cauchy's theorem, this is

\begin{eqnarray}\label{mainres4}
&=&\frac{1}{(k-1)!}\frac{\partial^{k-1}}{\partial s^{k-1}}\frac{(s-1)^k\zeta^k(s)W_{h,k,{\ell}}\left(s,x^A\right)}{s}\biggr\vert _{s=1}\nonumber\\
&=&\frac{1}{(k-1)!}\sum_{i=0}^{k-1}{k-1\choose i}\frac{\partial^{i}}{\partial s^{i}}W_{h,k,{\ell}}\left(s,x^A\right)\biggr\vert _{s=1}\frac{\partial^{k-1-i}}{\partial s^{k-1-i}}\frac{(s-1)^k\zeta^k(s)}{s}\biggr\vert _{s=1}\nonumber\\
&=&\sum_{i=0}^{k-1}\frac{1}{i!}\frac{\partial^{i}}{\partial s^{i}}W_{h,k,{\ell}}\left(s,x^A\right)\biggr\vert _{s=1}c_{k-1-i}(k),
\end{eqnarray}
where
\begin{eqnarray}
W_{A,h,k,{\ell}}\left(s,Q\right)=\sum_{q\leq Q}\phi_{h,k,{\ell}}(s,q) \nonumber
(q^{1/A}+h-\delta_A(q))^s.
\end{eqnarray}
By (\ref{conv}) we have 

\begin{eqnarray}\label{w}
\frac{\partial^i}{\partial s^i}W_{A,h,k,{\ell}}\left(s,Q\right)\biggr\vert _{s=1}&=&\frac{\partial^i}{\partial s^i}\sum_{d\leq Q}\varphi_{h,k,{\ell}}(d,s)\sum_{q\leq Q/d}\frac{d_{{\ell}-1}(q)((qd)^{1/A}+h-\delta_A(q))^s}{q}\biggr\vert _{s=1}\nonumber\\
&=&\sum_{j=0}^i{i\choose j}\sum_{d\leq Q}   \frac{\partial^{i-j}}{\partial s^{i-j}}   \varphi_{h,k,{\ell}}(d,s)\frac{\partial^{j}}{\partial s^{j}}V_{A,h,{\ell},Q}(d,s)\biggr\vert _{s=1},
\end{eqnarray}
where 
\begin{equation}
V_{A,h,{\ell},Q}(d,s)=\sum_{q\leq Q/d}\frac{d_{{\ell}-1}(q)((qd)^{1/A}+h-\delta_A(q))^s}{q}\nonumber
\end{equation}
and
\begin{eqnarray}\label{part}
\frac{\partial^{j}}{\partial s^{j}}V_{A,h,{\ell},Q}(d,s)\biggr\vert _{s=1}&=&(-1)^j\sum_{q\leq Q/d}\frac{d_{{\ell}-1}(q)((qd)^{1/A}+h-\delta_A(q))\log^j((qd)^{1/A}+h-\delta_A(q))}{q}
\nonumber\\
&=&(-A)^{-j}d^{1/A}\sum_{q\leq Q/d}\frac{d_{{\ell}-1}(q)\log^j(qd)}{q^{1-1/A}}+O_{A,h,j,{\ell}}\left((Q/d)^{\epsilon}\right)\nonumber\\
&=&(-A)^{-j}d^{1/A}\sum_{m=0}^j{j\choose m}\log^{j-m}d\sum_{q\leq Q/d}\frac{d_{{\ell}-1}(q)\log^{m}q}{q^{1-1/A}}+O_{A,h,j,{\ell}}\left((Q/d)^{\epsilon}\right).\nonumber\\
\end{eqnarray}
The inner summation on the right hand side of (\ref{part}) may be written as 
\begin{eqnarray}
\sum_{q\leq Q/d}\frac{d_{{\ell}-1}(q)\log^{m}q}{q^{1-1/A}}=\frac{(-1)^m}{2\pi i}\int_{(\epsilon)}\frac{d^m}{dw^m}\zeta^{{\ell}-1}(w+1)\frac{(Q/d)^{w+1/A}dw}{w+1/A},\nonumber
\end{eqnarray}
which may be evaluated using Cauchy's Theorem and classical results on the error term in the generalised Dirichlet divisor problem (Titchmarsh \cite{Titch}, Section 12). The error term is $O_{A,h,{\ell}}\left((Q/d)^{1/A-2/{\ell}+\epsilon}\right)$, and the residue at the pole at $w=0$ is 
\begin{eqnarray}\label{po}
&=&\frac{(-1)^m}{(m+{\ell}-2)!}\sum_{v=0}^{m+{\ell}-2}{m+{\ell}-2 \choose v}\frac{\partial^{v}}{\partial w^{v}}w^{m+{\ell}-1}\frac{\partial^{m}}{\partial w^{m}}\zeta^{{\ell}-1}(w+1)\biggr\vert _{w=0}\frac{\partial^{m+{\ell}-2-v}}{\partial w^{m+{\ell}-2-v}}\frac{(Q/d)^{w+1/A}}{w+1/A}\biggr\vert _{w=0}
\nonumber\\&=&(-1)^{m-1}(Q/d)^{1/A}\sum_{v=0}^{{\ell}+m-2}\frac{a_v({\ell}-1)(v-{\ell}+1)_m}{v!}\sum_{r=0}^{{\ell}+m-2-v}\frac{(-A)^{{\ell}+m-1-v-r}\log^r(Q/d)}{r!}\nonumber\\
&=&(-1)^{m-1}(Q/d)^{1/A}\sum_{r=0}^{{\ell}+m-2}\frac{\log^r(Q/d)}{r!}\sum_{v=0}^{{\ell}+m-2-r}\frac{(-A)^{{\ell}+m-1-v-r}a_{v}({\ell}-1)(v-{\ell}+1)_m}{v!}.\nonumber
\end{eqnarray}
As such, (\ref{part}) is
\begin{eqnarray}\label{dddd}
&=&(-A)^{-j}Q^{1/A}\sum_{m=0}^j{j\choose m}(-1)^{m-1}\log^{j-m}d
\sum_{r=0}^{{\ell}+m-2}\frac{\log^r(Q/d)}{r!}\nonumber\\
&\times & \sum_{v=0}^{{\ell}+m-2-r}\frac{(-A)^{{\ell}+m-1-v-r}a_{v}({\ell}-1)(v-{\ell}+1)_m}{v!}\\
&+& O_{A,h,j,{\ell}}\left(Q^{1/A-2/{\ell}+\epsilon}d^{2/{\ell}-\epsilon}\right)\nonumber
\end{eqnarray}
and so, expanding $\log^j(Q/d)$ as a polynomial in $\log Q$ on the right hand side of (\ref{dddd}), it follows that (\ref{w}) is 
\begin{eqnarray}\label{pof}
&=&-Q^{1/A}\sum_{j=0}^{i}{i\choose j}(-1)^j \sum_{m=0}^j {j\choose m}\sum_{r=0}^{{\ell}+m-2}\sum_{u=0}^{r}\frac{\log^uQ}{u!(r-u)!}\sum_{v=0}^{{\ell}+m-2-r}
\nonumber\\
&\times&  \frac{(-A)^{{\ell}+m-1-j-v-r}a_{v}({\ell}-1)(v-{\ell}+1)_m}{v!}
 \frac{\partial^{i-j}}{\partial s^{i-j}} \sum_{d\leq Q}\varphi_{h,k,{\ell}}^{}(d,s)(-\log d)^{j-m+r-u}\biggr\vert _{s=1}
\nonumber\\&+&O\left(Q^{1/A+\epsilon-2/{\ell}}\right)\nonumber
\end{eqnarray}
\begin{eqnarray}
&=&-Q^{1/A}\sum_{u=2-{\ell}}^{i}\frac{\log^{{\ell}+u-2} Q}{({\ell}+u-2)!}\sum_{j=u}^{i}{i\choose j}(-1)^j\sum_{m=u}^{j}{j\choose m}\sum_{r=0}^{m-u}\sum_{v=0}^{r}\nonumber\\
&\times&\frac{(-A)^{r-j-v+1}a_v({\ell}-1)(v-{\ell}+1)_m}{(m-r-u)!v!}
 \frac{\partial^{i-j}}{\partial s^{i-j}} \frac{\partial^{j+{\ell}-r-u-2}}{\partial w^{j+{\ell}-r-u-2}} \sum_{d\leq Q}\frac{\varphi_{h,k,{\ell}}^{}(d,s)}{d^w}\biggr\vert _{s=1,v=0}
 \nonumber\\&+&O\left(Q^{1/A+\epsilon-2/{\ell}}\right)\nonumber\\ 
&=&-Q^{1/A}\sum_{u=2-{\ell}}^{i}\frac{\log^{{\ell}+u-2} Q}{({\ell}+u-2)!}\sum_{j=u}^{i}{i\choose j}(-1)^j\sum_{m=u}^{j}{j\choose m}\sum_{r=u}^{m}\sum_{v=0}^{r-u}\nonumber\\
&\times&\frac{(-A)^{r-j-u-v+1}a_v({\ell}-1)(v-{\ell}+1)_m}{(m-r)!v!}
 \frac{\partial^{i-j}}{\partial s^{i-j}} \frac{\partial^{j+{\ell}-r-2}}{\partial w^{j+{\ell}-r-2}} \sum_{d\leq Q}\frac{\varphi_{h,k,{\ell}}^{}(d,s)}{d^w}\biggr\vert _{s=1,v=0}
\nonumber\\&+&O\left(Q^{1/A+\epsilon-2/{\ell}}\right)\nonumber
\end{eqnarray}

\begin{eqnarray}
&=&-Q^{1/A}\sum_{u=2-{\ell}}^{i}\frac{\log^{{\ell}+u-2} Q}{({\ell}+u-2)!}\sum_{j=u}^{i}{i\choose j}(-1)^j\sum_{r=u}^{j}\sum_{v=0}^{r-u}\frac{(-A)^{r-j-u-v+1}a_v({\ell}-1)}{v!}\nonumber\\
&\times&\left(\sum_{m=r}^{j}{j\choose m}\frac{(v-{\ell}+1)_m}{(m-r)!}\right)
 \frac{\partial^{i-j}}{\partial s^{i-j}} \frac{\partial^{j+{\ell}-r-2}}{\partial w^{j+{\ell}-r-2}} \sum_{d\leq Q}\frac{\varphi_{h,k,{\ell}}^{}(d,s)}{d^w}\biggr\vert _{s=1,v=0}
+O\left(Q^{1/A+\epsilon-2/l}\right)\nonumber\\ 
&=&-Q^{1/A}\sum_{u=2-{\ell}}^{i}\frac{\log^{{\ell}+u-2} Q}{({\ell}+u-2)!}\sum_{j=u}^{i}{i\choose j}\sum_{r=u}^{j}\sum_{v=0}^{r-u}\frac{(-A)^{r-j-u-v+1}a_v({\ell}-1)(v-{\ell}+1)_r}{v!}\nonumber\\
&\times&{{\ell}-v-2\choose j-r}
 \frac{\partial^{i-j}}{\partial s^{i-j}} \frac{\partial^{j+{\ell}-r-2}}{\partial w^{j+{\ell}-r-2}} \sum_{d\leq Q}\frac{\varphi_{h,k,{\ell}}^{}(d,s)}{d^w}\biggr\vert _{s=1,v=0}
+O\left(Q^{1/A+\epsilon-2/{\ell}}\right)\nonumber 
\end{eqnarray}
so (\ref{mainres4}) is

\begin{eqnarray}\label{polycalc2}
&=&-Q^{1/A}\sum_{u=2-{\ell}}^{k-1}\frac{\log^{{\ell}+u-2} Q}{({\ell}+u-2)!}   \sum_{j=u}^{k-1}{i\choose j}\sum_{r=u}^{j}\sum_{v=0}^{r-u}\frac{(-A)^{r-j-u-v+1}a_v({\ell}-1)(v-{\ell}+1)_r}{v!}\nonumber\\
&\times&{{\ell}-v-2\choose j-r}
 \sum_{i=j}^{k-1}\frac{c_{k-1-i}(k)}{i!} \frac{\partial^{i-j}}{\partial s^{i-j}} \frac{\partial^{j+{\ell}-r-2}}{\partial w^{j+{\ell}-r-2}} \sum_{d\leq Q}\frac{\varphi_{h,k,{\ell}}^{}(d,s)}{d^w}\biggr\vert _{s=1,v=0}
+O\left(Q^{1/A+\epsilon-2/{\ell}}\right)\nonumber \\
&=&-Q^{1/A}\sum_{u=0}^{k+{\ell}-3}\frac{\log^{u} Q}{u!}   \sum_{j=u-{\ell}+2}^{k-1}{i\choose j}\sum_{r=u-{\ell}+2}^{j}\sum_{v=0}^{r-u+{\ell}-2}\frac{(-A)^{r-j-u-v+{\ell}-1}a_v({\ell}-1)(v-{\ell}+1)_r}{v!}\nonumber\\
&\times&{{\ell}-v-2\choose j-r}
 \sum_{i=j}^{k-1}\frac{c_{k-1-i}(k)}{i!} \frac{\partial^{i-j}}{\partial s^{i-j}} \frac{\partial^{j+{\ell}-r-2}}{\partial w^{j+{\ell}-r-2}} \sum_{d\leq Q}\frac{\varphi_{h,k,{\ell}}^{}(d,s)}{d^w}\biggr\vert _{s=1,v=0}
+O\left(Q^{1/A+\epsilon-2/{\ell}}\right).\nonumber\\
\end{eqnarray}
Taking $Q=x^A$ in (\ref{polycalc2}) yields 

\begin{eqnarray}\label{ftm}
\frac{1}{(k-1)!}\frac{\partial^{k-1}}{\partial s^{k-1}}\frac{(s-1)^k\zeta^k(s)W_{h,k,{\ell}}\left(s,x^A\right)}{s}\biggr\vert _{s=1}=-x\sum_{u=0}^{k+{\ell}-3}\frac{a_{A,h,k,{\ell},u}\log^{u} x}{u!}+O\left(x^{1+\epsilon-2A/l}\right).\nonumber \\
\end{eqnarray}
From (\ref{p2}), (\ref{ftl}) and (\ref{ftm}), for $A<\theta_k$ we have
\begin{eqnarray}\label{polydef}
\sum_{n\leq x}d_k(n+h)d_{\ell}(n,A)&=&x\sum_{m=0}^{k-1}\sum_{n=0}^{{\ell}-1}\frac {A^nb_{h,k,{\ell},m,n}}{m!n!}\log^{m+n}x+x\sum_{m=0}^{k+{\ell}-3}\frac{a_{A,h,k,{\ell},m}\log^mx}{m!}\nonumber\\&+&O_{A,h,k,{\ell}}\left(x^{1+\epsilon-2A/{\ell}}\right)+O_{A,\delta,k}\left(x^{1-\delta}\right),
\end{eqnarray}
where the coefficients $b_{h,k,{\ell},m,n}$ and $a_{A,h,k,{\ell},u}$ are defined in  Section \ref{coeffs}.

\section{Proof of Theorem \ref{anotherpoly}}\label{pots}

This is a consequence of Theorem \ref{arith} and the method of proof of Theorem \ref{mainlem}. We have
\begin{eqnarray}\label{pota}
\sum_{n\leq x}d_k(n+h,A)d_{\ell}(n,B)&=&\sum_{q\leq x^B}d_{{\ell}-1}(q)\sum_{\substack{n\equiv h\pmod q\\q^{1/B}\leq n\leq x+h}}d_k(n,A) 
\nonumber\\
&=&A^{k-1}\sum_{q\leq x^B}d_{{\ell}-1}(q)\sum_{\substack{n\equiv h\pmod q\\q^{1/B}\leq n\leq x+h}}d_k(n)\nonumber\\ &+&O_{A,B,h,k}\left(x\log^{k-2}x\sum_{q\leq x^B}\frac{d_{{\ell}-1}(q)}{q}\right)   
\end{eqnarray}
provided that $B< \min(\theta_k,A\theta_{k-1})$, by Theorem \ref{arith}. The first term on the right hand side of (\ref{pota}) is identical to (\ref{p1}), and the summation in the error term is $O(\log^{{\ell}-1}x)$.

\section{Proof of Theorem \ref{t3}}\label{pot3}
We begin by establishing the analytic continuation of $\mathcal{D}_{h,k,{\ell}}(s,Q)$. We have 
\begin{eqnarray}
\mathcal{D}_{h,k,{\ell}}(s,Q)&=&\sum_{1}^{\infty}\frac{d_k(n+h)d_{\ell}\left(n,\frac{Q}{\log n}\right)}{(n+h)^s}\nonumber\\
&=&\sum_{1}^{\infty}\frac{d_k(n+h)}{(n+h)^s}\sum_{\substack{q\leq e^Q\\q|n}}d_{{\ell}-1}(q)\nonumber\\
&=& \sum_{q\leq e^Q}d_{{\ell}-1}(q)\sum_{\substack{n\equiv h\pmod q\\n>h}}\frac{d_k(n)}{n^s}\nonumber\\
&=&\sum_{q\leq e^Q}d_{{\ell}-1}(q)\sum_{n\equiv h\pmod q}\frac{d_k(n)}{n^s}-\sum_{q\leq e^Q}d_{{\ell}-1}(q)\sum_{\substack{n\equiv h\pmod q\\n\leq h}}\frac{d_k(n)}{n^s}.\nonumber
\end{eqnarray}
Since
\begin{eqnarray}
\sum_{n\equiv h\pmod q}\frac{d_k(n)}{n^s}=\frac{1}{\phi \left(\frac{q}{g}\right)}\sum_{\chi \left(\textrm{mod } \frac{q}{g}\right)}\overline{\chi}\left(\frac{h}{g}\right)\sum_{n=1}^{\infty}\frac{\chi(n)d_k(gn)}{(gn)^s}\nonumber
\end{eqnarray}
where 
\begin{eqnarray}\label{g}
\sum_{1}^{\infty}\frac{\chi(n)d_k(gn)}{(gn)^s}&=&\prod_p\sum_{\beta=0}^{\infty}d_k(p^{\beta+\delta})\chi(p^{\beta})p^{-(\beta+\delta)s}
\nonumber\\&=&L^k(s,\chi)b_k(s,\chi,g)
\end{eqnarray}
 is a meromorphic function of $s$ for all $h,k,{\ell}$, the above shows that $\mathcal{D}_{h,k,{\ell}}(s,Q)$ is a meromorphic function of $s$  for all $h,k,{\ell},Q$. We now observe that
\begin{eqnarray}\label{9}
\mathcal{D}_{h,k,{\ell}}(s,Q)=\zeta^k(s)Z_{h,k,{\ell}}\left(s,e^Q\right) +B_{h,k,{\ell}}(s,Q)
\end{eqnarray}
say, where 
$Z_{h,k,{\ell}}(s,Q)$ is defined in (\ref{sum}) and $B_{h,k,{\ell}}(s,Q)$ is an analytic function of $s$ for all fixed $h,k,{\ell},Q$. We also set
\begin{eqnarray}
D_{h,k,{\ell}}(x,Q)=\sum_{n\leq x}d_k(n+h)d_{\ell}\left(n,\frac{Q}{\log n}\right)\nonumber
\end{eqnarray}
and note that 
\begin{eqnarray}\label{mat}
\mathcal{D}_{h,k,{\ell}}(s,Q)=s\int_1^{\infty}D_{h,k,{\ell}}(x,Q)\frac{dx}{(x+h)^{s+1}}.\nonumber
\end{eqnarray}
By (\ref{9}), we have 
\begin{eqnarray}\label{pole}
\mathcal{D}_{h,k,l}(s,Q)=\frac{Z_{h,k,{\ell}}\left(s,e^Q\right)}{(s-1)^k}+C_{h,k,{\ell}}(s,Q)
\end{eqnarray}
for $\sigma>1$, where $C_{h,k,{\ell}}(s,Q)=O_{h,k,{\ell},Q}((s-1)^{1-k})$ as $s\rightarrow 1$. 
By  (\ref{9}) we know that $\mathcal{D}_{h,k,{\ell}}(1+it,Q)$ is continuous for $t\in\mathbb{R}$, $t\neq 0$. The Delange-Ikehara Tauberian theorem \cite{Del} then states that the behaviour of $D_{h,k,l}(x,Q)$ as $x\rightarrow\infty$ is determined by (\ref{pole}), specifically
\begin{eqnarray}\label{asy2}
\lim_{x\rightarrow\infty}\frac{D_{h,k,{\ell}}(x,Q)}{x\log^{k-1}x}=Z_{h,k,{\ell}}\left(1,e^Q\right).\nonumber
\end{eqnarray}\\

Arguing in the same way as in the proof of Theorem \ref{mainlem}, we now note that
\begin{eqnarray}\label{concho}
\frac{Z_{h,k,{\ell}}\left(1,e^Q\right)}{Q^{{\ell}-1}}=\frac{C_{k,{\ell}}f_{k,{\ell}}(h)}{(k-1)!({\ell}-1)!}+O_{h,k,{\ell}}\left(\frac{1}{Q}\right)
\end{eqnarray}
and suppose that $\lim_{y\rightarrow\infty}\mathcal{D}_{h,k,{\ell}}(1+it,y)$
is continuous when $t\neq 0$. Noting that $D_{h,k,{\ell}}(x,\log x)=D_{h,k,{\ell}}(x)$ and taking $Q=\log x$, by the Delange-Ikehara theorem
and (\ref{concho}) we have
\begin{eqnarray}
\frac{D_{h,k,{\ell}}(x)}{x\log^{k+\ell-2}x}\sim\frac{Z_{h,k,{\ell}}(1,x)}{\log^{{\ell}-1} x}=\frac{C_{k,{\ell}}f_{k,{\ell}}(h)}{(k-1)!({\ell}-1)!}+O_{h,k,{\ell}}\left(\frac{1}{\log x}\right)\nonumber
\end{eqnarray}
as $x\rightarrow\infty$.

\section{The coefficients in the case $k={\ell}=2$}\label{append}
\noindent We conclude by computing Estermann's \cite{Est} asymptotic expansion for $D_{h,2,2}(x)$ using  Theorem \ref{mainlem}. Thus
 \begin{eqnarray}\label{polynom2}
\sum_{n\leq x}d_2(n+h)d_2(n,A)&=&x\sum_{m=0}^{1}\sum_{n=0}^{1}\frac {A^nb_{h,2,2,m,n}}{m!n!}\log^{m+n}x+x\sum_{m=0}^{1}\frac{a_{A,h,2,2,m}\log^mx}{m!}+O_{A,h}\left(x^{1-\delta}\right)\nonumber\\
&=&Ab_{h,2,2,1,1}x\log^2x+\left(b_{h,2,2,1,0}+Ab_{h,2,2,0,1}+a_{A,h,2,2,1}\right)x\log x\nonumber\\
&+&\left(b_{h,2,2,0,0}+a_{A,h,2,2,0}\right)x+O_{A,h}\left(x^{1-\delta}\right).
\end{eqnarray}
Thus, putting $A=1/2$ and using the symmetry of the divisors of $n$ about $n^{1/2}$ in  (\ref{polynom2}), we obtain 
 \begin{eqnarray}\label{polynom3}
D_{h,2,2}(x)&=&b_{h,2,2,1,1}x\log^2x+\left(2b_{h,2,2,1,0}+b_{h,2,2,0,1}+2a_{1/2,h,2,2,1}\right)x\log x\nonumber\\
&+&2\left(b_{h,2,2,0,0}+a_{1/2,h,2,2,0}\right)x+O_{h}\left(x^{1-\delta}\right).
\end{eqnarray}
We now use Definitions \ref{form} and \ref{forma} to calculate the coefficients in (\ref{polynom3}). We adopt Estermann's notation 
\begin{eqnarray}
\sigma_{-1}'(h)=\sum_{d|h}\frac{\log d}{d}, \hspace{0.8cm} \sigma_{-1}''(h)=\sum_{d|h}\frac{\log^2 d}{d}
\end{eqnarray}
and
\begin{eqnarray}
a'=-\sum_{2}^{\infty}\frac{\mu(n)\log n}{n^2}, \hspace{0.8cm} a''=\sum_{2}^{\infty}\frac{\mu(n)\log^2 n}{n^2}.
\end{eqnarray}
Firstly, for the coefficient of $x\log^2 x$ we have 
\begin{eqnarray}\label{e2}
b_{h,2,2,1,1}=C_{2,2}(1,0)f_{h,2,2}(1,0)=\frac{6}{\pi^2}\sigma_{-1}(h).\nonumber
\end{eqnarray}
Secondly, for the coefficient of $x\log x$, we have 
 \begin{eqnarray}\label{e1}
&&2b_{h,2,2,1,0}+b_{h,2,2,0,1}+2a_{1/2,h,2,2,1}\nonumber\\
&=&2a_1(1)c_0(2)C_{2,2}(1,0)f_{h,2,2}(1,0)+2a_0(1)c_0(2) \frac{\partial}{\partial w}C_{2,2}(1,w)f_{h,2,2}(1,w)\biggr\vert _{w=0}\nonumber\\
&+&a_0(1)c_1(2)C_{2,2}(1,0)f_{h,2,2}(1,0)+a_0(1)c_0(2)\frac{\partial}{\partial s}C_{2,2}(s,0)f_{h,2,2}(s,0)\biggr\vert _{s=0}\nonumber\\
&-&a_0(1)c_0(2)C_{2,2}(1,0)f_{h,2,2}(1,0)\nonumber
\end{eqnarray}
\begin{eqnarray}
&=&2\gamma C_{2,2}(1,0)f_{h,2,2}(1,0)+2 \frac{\partial}{\partial w}C_{2,2}(1,w)f_{h,2,2}(1,w)\biggr\vert _{w=0}\nonumber\\
&+&(2\gamma-1)C_{2,2}(1,0)f_{h,2,2}(1,0)+\frac{\partial}{\partial s}C_{2,2}(s,0)f_{h,2,2}(s,0)\biggr\vert _{s=0}\nonumber\\
&-&C_{2,2}(1,0)f_{h,2,2}(1,0)\nonumber\\
&=&\frac{12}{\pi^2}(2\gamma-1)\sigma_{-1}(h)+2 \frac{\partial}{\partial w}C_{2,2}(1,w)f_{h,2,2}(1,w)\biggr\vert _{w=0}\nonumber\\
&+&\frac{\partial}{\partial s}C_{2,2}(s,0)f_{h,2,2}(s,0)\biggr\vert _{s=0}\nonumber\\
&=&\left(\frac{12}{\pi^2}(2\gamma-1)+2\frac{\partial}{\partial w}C_{2,2}(1,w)\biggr\vert _{w=0}+\frac{\partial}{\partial s}C_{2,2}(s,0)\biggr\vert _{s=0}\right)\sigma_{-1}(h)\nonumber\\
&+&\frac{6}{\pi^2}\left(2\frac{\partial}{\partial w}f_{h,2,2}(1,w)\biggr\vert _{w=0}+\frac{\partial}{\partial s}f_{h,2,2}(s,0)\biggr\vert _{s=0}\right)\nonumber\\
&=&\left(\frac{12}{\pi^2}(2\gamma-1)+4a'\right)\sigma_{-1}(h)-\frac{24}{\pi^2}\sigma_{-1}'(h).
\nonumber
\end{eqnarray}
Lastly, for the coefficient of $x$ we have 
\begin{eqnarray}\label{e0}
&&2b_{h,2,2,0,0}+2a_{1/2,h,2,2,0}\nonumber\\
&=&2a_1(1)c_1(2)C_{2,2}(1,0)f_{h,2,2}(1,0)+2a_1(1)C_{2,2}(1,0)\frac{\partial}{\partial s}f_{h,2,2}(s,0)\biggr\vert _{s=1}
\nonumber\\
&+&2a_1(1)f_{h,2,2}(1,0)\frac{\partial}{\partial s}C_{2,2}(s,0)\biggr\vert _{s=1}+ 2a_0(1)c_1(2)C_{2,2}(1,0)\frac{\partial}{\partial w}f_{h,2,2}(1,w)\biggr\vert _{w=0}\nonumber\\&+&2c_1(2)f_{h,2,2}(1,0)\frac{\partial}{\partial w}C_{2,2}(1,w)\biggr\vert _{w=0}+2C_{2,2}(1,0)\frac{\partial}{\partial s}\frac{\partial}{\partial w}f_{h,2,2}(s,w)\biggr\vert _{w=0,s=1}\nonumber\\
&+&2\frac{\partial}{\partial w}C_{2,2}(s,w)\biggr\vert _{w=0}\frac{\partial}{\partial s}f_{h,2,2}(s,w)\biggr\vert _{s=1}+2\frac{\partial}{\partial s}C_{2,2}(s,w)\biggr\vert _{s=1}\frac{\partial}{\partial w}f_{h,2,2}(s,w)\biggr\vert _{w=0}\nonumber
\end{eqnarray}
\begin{eqnarray}
&+&2f_{h,2,2}(1,0)\frac{\partial}{\partial w}\frac{\partial}{\partial s}C_{2,2}(s,w)\biggr\vert _{w=0,s=1}
\nonumber\\&-&c_1(2)C_{2,2}(1,0)f_{h,2,2}(1,0)-f_{h,2,2}(1,0)\frac{\partial}{\partial s} C_{2,2}(s,0)   \biggr\vert _{s=1}-C_{2,2}(1,0)\frac{\partial}{\partial s} f_{h,2,2}(s,0)   \biggr\vert _{s=1}\nonumber\\
&+&C_{2,2}(1,0)f_{h,2,2}(1,0)
\nonumber
\end{eqnarray}

\begin{eqnarray}
&=&\frac{12\gamma}{\pi^2}(2\gamma-1)\sigma_{-1}(h)-\frac{24\gamma}{\pi^2}\sigma_{-1}'(h)+4\gamma a'\sigma_{-1}(h)-\frac{12}{\pi^2}(2\gamma-1)\sigma_{-1}'(h)
+2(2\gamma-1)a'\sigma_{-1}(h)
\nonumber\\
&+&\frac{24}{\pi^2}\sigma_{-1}''(h)-8a'\sigma_{-1}'(h)+4a''\sigma_{-1}(h)-\frac{6}{\pi^2}(2\gamma-1)\sigma_{-1}(h)+\frac{12}{\pi^2}\sigma_{-1}'(h)-2a'\sigma_{-1}(h)+\frac{6}{\pi^2}\sigma_{-1}(h)
\nonumber\\
&=&\left(\frac{6}{\pi^2}(2\gamma-1)^2+\frac{6}{\pi^2}+4a'(2\gamma-1)+4a''   \right)\sigma_{-1}(h)-\left(\frac{24}{\pi^2}(2\gamma-1)+8a'\right)\sigma_{-1}'(h)+\frac{24}{\pi^2}\sigma_{-1}''(h).    \nonumber
\end{eqnarray}

\vspace{1cm}

\noindent \textbf{Acknowledgment:} We would like to thank the anonymous referee for their comments and suggestions regarding the paper. We would also like to thank Prof. Steve Gonek and Prof. Zeev Rudnick for their suggestions on a previous version of this paper, and Prof. Andrew Granville, Prof. Roger Heath-Brown, Prof. Chris Hughes, and Prof. Terence Tao for the many comments and ideas provided that helped improve the paper. The first author is also grateful to the Leverhulme Trust (RPG-2017-320) for the support through the research project grant ``Moments of L-functions in Function Fields and Random Matrix Theory". The second author is grateful for a PhD studentship supported by the Faculty of Environment, Science and Economy at the University of Exeter.

\end{document}